\documentclass[12pt]{article}
\usepackage{amsthm,amssymb,amsmath,graphicx,cite}
\usepackage{algpseudocode,algorithm,algorithmicx}

\usepackage[margin = 1in]{geometry}

\newcommand{\diam}{\mathsf{diam}}

\newcommand{\width}{\mathsf{width}}
\newcommand{\faces}{\text{\rm faces}}

\newcommand{\dmin}{\displaystyle\min}
\newcommand{\dinf}{\displaystyle\inf}
\newcommand{\dmax}{\displaystyle\max}

\newcommand{\1}{\mathbf{1}}

\newcommand{\ip}[2]{\left\langle #1 , #2 \right\rangle}    % inner product
\newcommand{\R}{{\mathbb R}}

\newcommand{\B}{{\mathbb B\,}}

\newcommand{\conv}{{\mathsf{conv}}}

\newcommand{\dist}{\mathsf{dist}}

\newcommand{\dom}{{\mathrm{dom}}}

\newcommand{\dargmin}{\displaystyle \argmin}

\DeclareMathOperator*{\Argmin}{Argmin}
\DeclareMathOperator*{\Argmax}{Argmax}
\DeclareMathOperator*{\argmin}{argmin}
\DeclareMathOperator*{\argmax}{argmax}

\newtheorem{lemma}{Lemma}
\newtheorem{theorem}{Theorem}
\newtheorem{proposition}{Proposition}

\newtheorem{corollary}{Corollary}

\newtheorem{example}{Example}
\newtheorem{definition}{Definition}
\def\transp{^{\text{\sf T}}}
\newcommand{\matr}[1]{\begin{bmatrix} #1 \end{bmatrix}}    % matrix 

\title{The condition of a function relative to a polytope} 
\author{
David H. Gutman\thanks{Department of Mathematical Sciences, 
Carnegie Mellon University, USA, {\tt dgutman@andrew.cmu.edu}}
\and
Javier F. Pe\~na\thanks{Tepper School of Business,
Carnegie Mellon University, USA, {\tt jfp@andrew.cmu.edu}}
}

\begin{document}

\maketitle

\begin{abstract}
The {\em condition number} of a smooth convex function, namely the ratio of its smoothness to strong convexity constants,  is closely tied to fundamental properties of the function.  In particular, 
the condition number of a quadratic convex function is precisely the square of the {\em diameter-to-width} ratio of a canonical ellipsoid associated to the function.  Furthermore, the condition number of a function bounds the linear rate of convergence of the gradient descent algorithm for unconstrained minimization. 

We propose a condition number of a smooth convex function relative to a {\em reference polytope.}  This relative condition number is defined as the ratio of a {\em relative smooth constant} to a {\em relative strong convexity constant} of the function, where both constants are relative to  the reference polytope.  The relative condition number extends the main properties of the traditional condition number.  In particular, we show that the condition number of a quadratic convex function relative to a polytope is precisely the square of the {\em diameter-to-facial-distance} ratio of a scaled polytope for a canonical scaling  induced by the function.  Furthermore, we illustrate how the relative condition number of a function bounds the linear rate of convergence of first-order methods for minimization of the function over the polytope.

\end{abstract}

%\newpage

%%%%%%%%%%%%%%%%%%%%%%%%%%%%%%%%%%%%%%%%%%%%%%%%%%%%%%%%%%%%%%%%%%%%%%%%%%%%%%%%%%%%%%%
%
%Introduction
%
%%%%%%%%%%%%%%%%%%%%%%%%%%%%%%%%%%%%%%%%%%%%%%%%%%%%%%%%%%%%%%%%%%%%%%%%%%%%%%%%%%%%%%%
\section{Introduction}
\label{sec.intro}

Let $f:\R^m\to\R\cup\{\infty\}$ be a convex function.  The {\em condition number} of $f$ is the ratio $\frac{L_f}{\mu_f}$ where  $L_f$ and $\mu_f$ are  respectively the {\em smoothness} and {\em strong convexity} constants of the function $f$ as detailed in Definition~\ref{def.regular} below.  The condition number $\frac{L_f}{\mu_f}$ is closely tied to a number of fundamental properties of the function $f$.  In the special case when $f$ is a quadratic convex function the condition number has the following geometric insight.  Suppose $f(u) = \frac{1}{2}\ip{Qu}{u} + \ip{b}{u}$ where $Q$ is symmetric and positive definite.  Then the condition number $\frac{L_f}{\mu_f}$ is precisely the square of the aspect ratio of the ellipsoid $Q^{1/2}\B:= 
\{Q^{1/2}u :u\in \R^m, \|u\|\le 1\}.$

The condition number $\frac{L_f}{\mu_f}$ also bounds the linear convergence rate of the gradient descent algorithm for the unconstrained minimization problem 
$$f^\star = \dmin_{u\in\R^m} f(u).$$ 
%is known to converge linearly with rate of convergence bounded by $\frac{L_f}{\mu_f}$.   
More precisely, for a suitable choice of step sizes the iterates $u_k, \; k=0,1,\dots$ generated by the gradient descent algorithm satisfy
\[
f(u_k) - f^\star \le \left(1 - \frac{\mu_f}{L_f} \right)^k(f(u_0) - f^\star).
\]
%An accelerated version of the gradient descent algorithm, like the one described in~\cite{Nest04}, generates iterates that satisfy
%\[
%f(u_k) - f^\star \le C\cdot\left(1 - \sqrt{\frac{\mu_f}{L_f}} \right)^k(f(u_0) - f^\star)
%f(u_k) - f^\star \le L_f\left(1 - \sqrt{\frac{\mu_f}{L_f}} \right)^k\|u_0- u^\star\|^2,
%\]
%where $u^\star = \argmin_{u\in \R^m} f(u)$.
%for a constant $C$.
The  references~\cite{BubeLS15,DrusFR16,KariV17,MaGV17,NecoNG18,Nest04,Nest13}, among others, discuss the above type of linear convergence and a number of interesting related developments.  In particular, Necoara, Nesterov and Glineur~\cite{NecoNG18} establish linear convergence properties for a wide class of first-order methods under assumptions that are weaker than strong convexity.

\medskip

Let $f:\R^m\to\R\cup\{\infty\}$ and $A\in\R^{m\times n}$ be such that $\conv(A)\subseteq \dom(f).$  We propose a relative smoothness constant $L_{f,A}$ and  a relative strong convexity constant $\mu_{f,A}$ of the function $f$ relative to the reference polytope $\conv(A)$.  See Definition~\ref{def.relative} below for details.    
Our main results highlight the tight connection between the relative constants and geometric features of the polytope $\conv(A)$.  In particular, we establish some interesting relationship between the relative smoothness and strong convexity constants and the diameter and {\em facial distance} of the polytope.  The facial distance  was introduced by Pe\~na and Rodr\'iguez~\cite{PenaR16} albeit in a more restricted context.  These relationships in turn enables us to show that the relative condition number $\frac{L_{f,A}}{\mu_{f,A}}$ extends some of the main properties of the classical condition number $\frac{L_{f}}{\mu_{f}}$.
In particular, we provide the following interesting geometric insight on the relative condition number when $f$ is quadratic.  Suppose  $f(u) = \frac{1}{2} \ip{Qu}{u} + \ip{b}{u}$ where $Q$ is symmetric and positive definite.  As we detail in Section~\ref{sec.relcond} below, in this case the relative condition number $\frac{L_{f,A}}{\mu_{f,A}}$ is precisely the square of the ratio of the diameter to the facial distance of the polytope $\conv(Q^{1/2}A)$.  For a general convex function, we show that the relative condition number $\frac{L_{f,A}}{\mu_{f,A}}$ can be bounded above by the product of the classical condition number $\frac{L_{f}}{\mu_{f}}$ and the square of the ratio of the diameter to the facial distance of $\conv(A)$.  The latter quantity can be seen as a kind of condition number of the polytope $\conv(A)$.

We also illustrate how the relative condition number bounds the  linear convergence rate of first-order methods for the minimization problem 
\begin{equation}\label{eq.Prob}
f^\star = \dmin_{u\in\conv(A)} f(u).
\end{equation}
More precisely, we show that the iterates generated by both the Frank-Wolfe algorithm with away steps and a version of the projected gradient algorithm have objective values that convergence linearly to $f^\star$ with linear convergence rate bounded by $\frac{L_{f,A}}{\mu_{f,A}}$.
We should note that the linear convergence of the Frank-Wolfe algorithm with away steps and the projected gradient algorithm, as well as of other first-order methods had been previously established in~\cite{BeckS15,LacoJ15,NecoNG18,Nest13,PenaR16} under various kinds of assumptions.  Our approach based on the relative condition number yields a proof of linear convergence for the Frank-Wolfe algorithm with away steps that is significantly shorter and simpler than the ones previously presented in~\cite{BeckS15,LacoJ15,PenaR16}.  Our approach also reveals  some simple ideas at the root of the proofs of linear convergence properties of these first-order algorithms.

The relative constants $L_{f,A}, \mu_{f,A}$ are defined {\em globally.}  In particular, they do not depend on any specific point in $\conv(A)$.   We also consider a version of the {\em quadratic functional growth constant} $\mu_{f,A}^\star$ following the construction of Necoara, Nesterov and Glineur~\cite[Definition 4]{NecoNG18}.  Unlike $\mu_{f,A}$, the constant $\mu_{f,A}^\star$  depends explicitly on the set of minimizers of $f$ on $\conv(A)$.  The constant $\mu_{f,A}^\star$ can be seen as a refinement of $\mu_{f,A}$.  It is always the case that $\mu_{f,A}^\star$ is larger, and can be quite a bit larger, than $\mu_{f,A}$.  Indeed, we show that for some important classes of non-strongly convex functions the constant $\mu_{f,A}^\star$ is positive while $\mu_{f,A}$ may not be.  (See Theorem~\ref{thm.loc}.)

Our work draws on and relates to the recent articles~\cite{BeckS15,LacoJ15,LuFN17,NecoNG18,PenaR16}.  Our  construction of $L_{f,A}$ and $\mu_{f,A}$ is in the spirit introduced by Lu, Freund, and Nesterov~\cite{LuFN17}.  Lu et al.~\cite{LuFN17}  extend the concepts of smoothness and strong convexity constants by considering them {\em relative} to a {\em reference} function, see~\cite[Definition 1.1 and 1.2]{LuFN17}.  %The relative constants become the classical ones when the reference function is the square of the underlying norm in $\R^m$.  
Our construction of $L_{f,A}$ and $\mu_{f,A}$ is also related to the {\em curvature constant} $C^A_f$ and {\em geometric strong convexity constant} $\mu^A_f$ proposed by Lacoste-Julien and Jaggi in~\cite[Appendix C]{LacoJ15}.  The quadratic functional growth constant, as well as other more restrictive growth constants, were proposed by Necoara, Nesterov, and Glineur~\cite{NecoNG18} to give conditions that ensure the linear convergence of first-order methods.  A similar quadratic growth approach was also used by Beck and Shtern~\cite{BeckS15} to established the linear convergence of a conditional gradient algorithm with away steps for non-strongly convex functions.  In contrast to the approaches in~\cite{BeckS15,LacoJ15,NecoNG18,PenaR16}, our construction of the relative constants applies to any choice of norm in $\R^n$.  Our results reveal interesting geometric insights when this norm is the $\ell_1$ norm.  Our construction of the relative constants $L_{f,A}, \; \mu_{f,A}, \; \mu_{f,A}^\star$ and all of our results concerning them scale appropriately, that is, they scale by $\lambda$ whenever the objective function $f$ is replaced by $\tilde f = \lambda f$ for some constant $\lambda > 0$.  In particular,  the relative condition number $\frac{L_{f,A}}{\mu_{f,A}}$ and all of our bounds on it are invariant under positive scaling of $f$.

The remainder of the paper is organized as follows. Section \ref{sec.relcond}, the main section of the paper, presents our central ideas and results, namely the construction of relative smoothness and strong convexity constants and their main properties. Section \ref{sec.algos} illustrates how the relative condition number bounds the linear rate of convergence of the Frank-Wolfe algorithm with away steps and of the projected gradient algorithm for problem~\eqref{eq.Prob}.

%Section \ref{sec.refine} discusses future refinements and extensions of the new condition numbers.

%%%%%%%%%%%%%%%%%%%%%%%%%%%%%%%%%%%%%%%%%%%%%%%%%%%%%%%%%%%%%%%%%%%%%%%%%%%%%%%%%%%%%%%
%
%Condition number relative to a polytope
%
%%%%%%%%%%%%%%%%%%%%%%%%%%%%%%%%%%%%%%%%%%%%%%%%%%%%%%%%%%%%%%%%%%%%%%%%%%%%%%%%%%%%%%%
\section{Condition number relative to a polytope}
\label{sec.relcond}

This section presents the central ideas of this paper.  We introduce the relative smoothness and relative strong convexity of a function relative to a polytope and establish their main properties.  We will use the following notation.  Let $\Delta_{n-1} \subseteq \R^n$ denote the standard simplex, that is,  $\Delta_{n-1} := \{x\in\R^n_+: \|x\|_1 = 1\}$.  For convenience, we will make the following slight abuse of notation.  For $A\in \R^{m\times n}$ we will also write $A$ to denote the set of columns of $A$.  The precise meaning of $A$ will be clear from the context.

For $A\in\R^{m\times n}$ consider the following polytope generated by $A$  $$\conv(A):=\{Ax:x\in\Delta_{n-1}\}.$$ 

\medskip

For $u\in \conv(A)$ let $Z(u) := \{z\in \Delta_{n-1}: Az = u\}$. Suppose $\R^n$ is endowed with a norm $\|\cdot\|$.  For $x\in \Delta_{n-1}$ and $u\in \conv(A)$ define $$\dist(x,Z(u)) := \dmin_{z\in Z(u)} \|x-z\|.$$

%
%The Relative Lipschitz and Strong Convexity Constants
%
\subsection{The relative Lipschitz and strong convexity constants}
\label{subsec.def}

To motivate our main construction we first recall the classical notion of smoothness and strong convexity constants.

\begin{definition}\label{def.regular}{\em Suppose $\R^m$ is endowed with a norm $\|\cdot\|$ and $f:\R^m \rightarrow \R \cup \{\infty\}$ is a differentiable convex function.  

\begin{itemize}
\item[(a)] The function $f$ is $L_f$-smooth on $S \subseteq \dom(f)$ for some constant $L_f>0$ if for all $u,v\in S$
\[
f(v) \le f(u) + \ip{\nabla f(u)}{v-u} + \frac{L_f}{2}\cdot \|v-u\|^2.
\]
\item[(b)] The function $f$ is $\mu_f$-strongly convex  on $S \subseteq \dom(f)$ for some constant $\mu_f \ge 0$ if for all $u,v\in S$
\[
f(v) \ge f(u) + \ip{\nabla f(u)}{v-u} + \frac{\mu_f}{2}\cdot \|v-u\|^2.
\]
\end{itemize}
}
\end{definition}

Next, we present our main construction.

\begin{definition}\label{def.relative}{\em Let $\R^m$ be endowed with a norm $\|\cdot\|$, $A\in \R^{m\times n}$ have at least two different columns, and $f:\R^m \rightarrow \R \cup \{\infty\}$ be a differentiable convex function such that $\conv(A) \subseteq \dom(f)$.

\begin{itemize}
\item[(a)] Define the smoothness constant $L_{f,A}$ of $f$ relative to $A$ as follows
\[
L_{f,A} := \sup_{u\in \conv(A)\atop x\in \Delta_{n-1}\setminus Z(u)} \frac{2(f(Ax) - f(u) - \ip{\nabla f(u)}{Ax-u})}{\dist(x,Z(u))^2}.
\]
\item[(b)] Define the strong convexity constant $\mu_{f,A}$  of $f$ relative to $A$ as follows
\[
\mu_{f,A} := \inf_{u\in \conv(A)\atop x\in \Delta_{n-1}\setminus Z(u)} \frac{2(f(Ax) - f(u) - \ip{\nabla f(u)}{Ax-u})}{\dist(x,Z(u))^2}.
\]

\end{itemize}
}
\end{definition}

The relative constants are natural extensions of the classical ones.  Observe that a differentiable function $f:\R^m\rightarrow \R\cup \{\infty\}$ is $L_f$-smooth and $\mu_f$-strongly convex on $S\subseteq \dom(f)$ if and only if for all $u,v\in S$
\[
\frac{\mu_f}{2}\cdot\|v-u\|^2\leq f(v)-f(u)- \ip{\nabla f(u)}{v-u}\leq\frac{L_f}{2}\cdot\|v-u\|^2.
\]
Likewise, a differentiable function $f:\R^m\rightarrow \R\cup \{\infty\}$ is $L_{f,A}$-smooth and $\mu_{f,A}$-strongly convex on $\conv(A)\subseteq \dom(f)$ for $A\in \R^{m\times n}$ if and only if for all $u\in \conv(A)$ and $x\in \Delta_{n-1}$
\[%\begin{equation}\label{eq.muLrel}
\frac{\mu_{f,A}}{2} \cdot \dist(x,Z(u))^2 \le f(Ax) - f(u) - \ip{\nabla f(u)}{Ax-u} \le \frac{L_{f,A}}{2} \cdot \dist(x,Z(u))^2.
\]%\end{equation}

The relative smoothness and strong convexity constants  $L_{f,A}$ and $\mu_{f,A}$ are closely related to the {\em curvature constant} $C^A_f$ and {\em geometric strong convexity constant} $\mu^A_f$ proposed by Lacoste-Julien and Jaggi in~\cite[Appendix C]{LacoJ15}.  However, the construction in~\cite[Appendix C]{LacoJ15} follows a fairly different path.  In particular, the definition of $\mu^A_f$ is tied to some variants of the Frank-Wolfe algorithm for problem~\eqref{eq.Prob}.  By contrast, our construction of $L_{f,A}$ and $\mu_{f,A}$ depends only on the pair $(f,A)$, applies to any norm in $\R^n$, and does not depend on any particular algorithm.  As we discuss in Section~\ref{sec.algos}, the relative condition number $\frac{L_{f,A}}{\mu_{f,A}}$  bounds the linear rates of convergence of the projected gradient algorithm and of the Frank-Wolfe algorithm with away steps for problem~\eqref{eq.Prob}.  
We also note that in the special case when $\R^n$ is endowed with the Euclidean norm $\ell_2$, the strong convexity constant $\mu_{f,A}$ is related to the quadratic gradient growth condition defined in~\cite{NecoNG18}.

%
%Geometric Properties of the Relative Constants
%
\subsection{Geometric properties of the relative constants}
\label{subsec.prop}
We next present some geometric properties of the constants $L_{f,A}$ and $\mu_{f,A}$. The properties below show that these constants are finite and positive when $f$ is $L_f$-smooth and $\mu_f$-strongly convex for some $L_f,\mu_f>0$.  The properties below also yield a nice analogy between the relative condition number $\frac{L_{f,A}}{\mu_{f,A}}$ and the usual condition number $\frac{L_f}{\mu_f}$. 

\medskip

Our results rely on the concept of facial distance introduced by Pe\~na and Rodr\'iguez~\cite{PenaR16}. Let $\R^m$ be endowed with a norm $\|\cdot\|$.  For $A\in \R^{m\times n}$ with at least two different columns  the {\em facial distance} $\Phi(A)$ is defined as follows
\[
\Phi(A) := \dmin_{F \in \faces(\conv(A))\atop \emptyset \ne F \ne \conv(A)} \dist(F,\conv(A\setminus F)).
\]
Here %$A\setminus F$ denotes the set of columns of $A$ that are not contained in $F$.  The expression 
$\faces(\conv(A))$ denotes the set of faces of $\conv(A)$ and $\dist(F,G) = \dinf_{u\in F, v\in G} \|u-v\|$ for nonempty $F,G\subseteq\R^m$.   Observe that $\Phi(A) > 0$ for all $A\in \R^{m\times n}$ with at least two different columns.

The following example illustrates the facial distance for two canonical polytopes, namely the standard simplex and the $\ell_1$ unit ball. 
% In Section~\ref{sec.spectratope} below we extend this example to two  related spectratopes, namely the spectraplex and the unit ball of the nuclear norm.

%The value of the facial distance in each case follow via a straightforward calculation.  (For further details, see~\cite{}.)
\begin{example}\label{example.polytopes}
{\em Suppose $m > 1$ and $\R^m$ is endowed with the Euclidean norm $\|\cdot\|_2$.

\begin{itemize}
\item[(a)] For $A = I_m \in \R^{m\times m}$ we have $\conv(A) = \Delta_{m-1}$. 
% and $\Phi(A) = \left\{\begin{array}{ll}
% \frac{2}{\sqrt{m}} & \text{ if } m \text{ is even } \\
% \frac{2}{\sqrt{m-\frac{1}{m}}}   & \text{ if } m \text{ is odd. } 
% \end{array} \right.$
In this case $\Phi(A)$ is attained at any face of $\conv(A)$ of dimension $k:=\lfloor{\frac{m}{2} }\rfloor $.  In particular for $F = 
\conv\{e_1,\dots,e_k\}$ we get
\begin{align*}
\Phi(A) &=
\dist(F,\conv(A\setminus F)) 
%\\ & 
= \left\|\frac{e_1+\cdots +e_k}{k} - \frac{e_{k+1}+\cdots +e_m}{m-k}\right\|_2 
\\ &= \sqrt{\frac{m}{k(m-k)}}
\\ &= \left\{\begin{array}{ll}
 \frac{2}{\sqrt{m}} & \text{ if } m \text{ is even } \\
 \frac{2}{\sqrt{m-\frac{1}{m}}}   & \text{ if } m \text{ is odd. } 
 \end{array} \right.
\end{align*}
 
\item[(b)] For $A = \matr{I_m & -I_m}$ we have $\conv(A) = \{u\in\R^m: \|u\|_1 \le 1\}$. % and $\Phi(A) = \frac{1}{\sqrt{m-1}}$.  
In this case $\Phi(A)$ is attained at any face of $\conv(A)$ of dimension $(m-2)$.  In particular, for $F = \conv\{e_1,\dots,e_{m-1}\}$ we get
\[
\Phi(A) =
\dist(F,\conv(A\setminus F)) = \left\|\frac{e_1+\cdots +e_{m-1}}{m-1} - 0\right\|_2  = \frac{1}{\sqrt{m-1}}.
\]
\end{itemize}
}
\end{example}

Some of the results below are stated in terms of the {\em diameter} of a set defined as follows.  For $A\subseteq \R^m$
\[
\diam(A):= \sup_{u,v\in A} \|u-v\|.
\]
The alternative characterizations of $\diam(A)$ and $\Phi(A)$ in the following proposition provide the crux for the geometric properties of $L_{f,A}$ and $\mu_{f,A}$.  We defer the proof of Proposition~\ref{prop.diam.dist} to Section~\ref{sec.proofs} since it relies on Lemma~\ref{lemma.error.bound.loc} below. 

\begin{proposition}\label{prop.diam.dist}
Suppose $\R^n$ is endowed with the $\ell_1$ norm $\|\cdot\|_1$. Then for all $A\in \R^{m\times n}$ with at least two different columns
\[
\diam(A) = \max_{u\in \conv(A)\atop x\in \Delta_{n-1}\setminus Z(u)} \frac{2\|Ax - u\|}{\dist(x,Z(u))},  \; 
\Phi(A) = \min_{u\in \conv(A)\atop x\in \Delta_{n-1}\setminus Z(u)} \frac{2\|Ax - u\|}{\dist(x,Z(u))}.
\] 
\end{proposition}
%\begin{proof}  This is a special case of Proposition~\ref{prop.diam.spec.dist} in Section~\ref{sec.spectratope}.
%\end{proof}

\begin{corollary}\label{the.corol}
Suppose $\R^n$ and $\R^m$ are respectively endowed with the $\ell_1$ norm $\|\cdot\|_1$ and the $\ell_2$ norm $\|\cdot\|_2$, and $f(u) = \frac{1}{2} \ip{Qu}{u}+\ip{b}{u}$ for some $b\in\R^m$ and $Q\in\R^{m\times m}$ symmetric and positive definite.  Then for all $A\in \R^{m\times n}$ with at least two different columns
\[
L_{f,A} = \frac{\diam(Q^{1/2}A)^2}{4}, \; \mu_{f,A} = \frac{\Phi(Q^{1/2}A)^2}{4}.
\] 
In particular 
\begin{equation}\label{cond.rel}
\frac{L_{f,A}}{\mu_{f,A}} = \frac{\diam(Q^{1/2}A)^2}{\Phi(Q^{1/2}A)^2}.
\end{equation}
\end{corollary}

Corollary~\ref{the.corol} yields the following  analogy between the relative condition number $\frac{L_{f,A}}{\mu_{f,A}}$ and the usual condition number $\frac{L_f}{\mu_f}$ of a strongly convex quadratic function $f$.  Under the assumptions of Corollary~\ref{the.corol} it readily follows that
\[
L_f = \lambda_{\max}(Q) = \lambda_{\max}(Q^{1/2})^2, \;\; \mu_f = \lambda_{\min}(Q) = \lambda_{\min}(Q^{1/2})^2.
\]
Furthermore, observe that $2\lambda_{\max}(Q^{1/2})$ and $2\lambda_{\min}(Q^{1/2})$ are respectively the diameter (length of longest principal axis) and the width (length of shortest principal axis) of the ellipsoid $\{Q^{1/2} u: \|u\|_2 \le 1\} = Q^{1/2} \B,$ where $\B := \{u\in\R^m: \|u\|_2\le 1\}.$ Therefore,
\[
L_f = \frac{\diam(Q^{1/2}\B)^2}{4}, \;\; \mu_f = \frac{\width(Q^{1/2}\B)^2}{4}.
\]
In particular,
\begin{equation}\label{cond.usual}
\frac{L_f}{\mu_f} = \frac{\diam(Q^{1/2}\mathbb{B})^2}{\width(Q^{1/2}\mathbb{B})^2}.
\end{equation}
Observe the striking resemblance between~\eqref{cond.rel} and \eqref{cond.usual}.

\begin{corollary}\label{prop.bounds} 
Let $\R^n$ be endowed with the $\ell_1$ norm $\|\cdot\|_1$.
Suppose $A \in \R^{m\times n}$ has at least two different columns and $f:\R^m \rightarrow \R\cup\{\infty\}$ is $L_f$-smooth and $\mu_f$-strongly convex  on $\conv(A)$. Then
\[
L_{f,A} \le \frac{L_f \cdot\diam(A)^2}{4}, \; \mu_{f,A} \ge \frac{\mu_f \cdot \Phi(A)^2}{4}.
\] 
In particular, $$\frac{L_{f,A}}{\mu_{f,A}} \le\frac{L_f}{\mu_f}\cdot \frac{\diam(A)^2}{\Phi(A)^2}.$$
\end{corollary}

The following proposition gives an identity and bound similar to those in Proposition~\ref{prop.diam.dist} for the general case when $\R^n$ is endowed with an arbitrary norm.  We defer the proof of Proposition~\ref{prop.dist.gral} to Section~\ref{sec.proofs}.

\begin{proposition}\label{prop.dist.gral}
Let $\R^n$ and $\R^m$ be endowed with arbitrary norms and $A \in \R^{m\times n}$ have at least two different columns. Then 
\begin{equation}\label{eq.general.diam}
\max_{w\in \R^n \setminus \{0\}\atop \ip{\1}{w} = 0}
\frac{\|Aw\|}{\|w\|} = \max_{u\in \conv(A)\atop x\in \Delta_{n-1}\setminus Z(u)} \frac{\|Ax - u\|}{\dist(x,Z(u))},
\end{equation}
and
\begin{equation}\label{eq.general.Phi}
\frac{\Phi(A)}{\dmax_{i=1,\dots,n} \|e_i\|} \le \dmin_{u\in \conv(A)\atop x\in \Delta_{n-1}\setminus Z(u)} \frac{2\|Ax - v\|}{\dist(x,Z(u))}.
\end{equation} 
\end{proposition}

\begin{corollary}\label{prop.bounds.gral} 
Let $\R^n$ and $\R^m$ be endowed with arbitrary norms.  
Suppose $A \in \R^{m\times n}$ has at least two different columns and $f:\R^m \rightarrow \R\cup\{\infty\}$ is $L_f$-smooth and $\mu_f$-strongly convex on $\conv(A)$.  Then
\[%\begin{equation}\label{eq.general.cond}
L_{f,A} \le 
L_f \cdot \dmax_{w\in \R^n \setminus \{0\}\atop \ip{\1}{w} = 0}
\frac{\|Aw\|^2}{\|w\|^2}, \;\;
\mu_{f,A} \ge \frac{\mu_f \cdot \Phi(A)^2}{4 \dmax_{i=1,\dots,n} \|e_i\|^2}, \; 
\]%\end{equation}
In particular, $$\frac{L_{f,A}}{\mu_{f,A}} \le \frac{L_f}{\mu_f}\cdot\frac{4\dmax_{i=1,\dots,n} \|e_i\|^2}{\Phi(A)^2}\cdot\dmax_{w\in \R^n \setminus \{0\}\atop \ip{\1}{w} = 0}
\frac{\|Aw\|^2}{\|w\|^2}.$$
\end{corollary}

%
%The Local Strong Convexity Constant
%
\subsection{A refinement of the relative strong convexity constant}
\label{subsec.loc}
The construction of the constants $L_{f,A}$ and $\mu_{f,A}$ is {\em global} as it depends on $f$ and the entire set $\conv(A)$.   We next describe a {\em local} refinement of the strong convexity constant.  
%As we detail in Section~\ref{sec.algos} below, this refinement in turn allows us to state sharper convergence bounds of first-order methods for polytope constrained optimization.  

\begin{definition}\label{def.local}{\em
Let $\R^m$ be endowed with a norm $\|\cdot\|$, $A\in \R^{m\times n}$ have at least two different columns, and $f:\R^m \rightarrow \R \cup \{\infty\}$ be a differentiable convex function such that $\conv(A) \subseteq \dom(f)$.   Let  $f^\star:= \dmin_{u\in \conv(A)} f(u) = \dmin_{x\in \Delta_{n-1}} f(Ax),\; Z^\star:= \{z\in \Delta_{n-1}: f(Az) = f^\star\}.$ Following~\cite{NecoNG18}, define the quadratic functional growth constant $\mu_{f,A}^\star$ as 
\begin{equation}\label{eq.strong.conv.local}
\mu_{f,A}^\star := \inf_{ x\in \Delta_{n-1}\setminus Z^\star} \frac{2(f(Ax) - f^\star)}{\dist(x,Z^\star)^2}.
\end{equation}
}
\end{definition}

The convexity of $f$ readily implies $\mu_{f,A}^\star \ge \mu_{f,A} \ge 0$. Furthermore, as we next discuss, for an important class of functions $\mu_{f,A}^\star$ is positive while $\mu_{f,A}$ may not be.  Suppose $f:\R^m\rightarrow\R\cup\{\infty\}$ is defined as  $f(u) = g(Eu) + \ip{b}{u}$ 
where $g:\R^p\rightarrow \R\cup\{\infty\}$ is a strongly convex function, 
$b\in \R^m$, and $E\in \R^{p\times m}$.  A  function $f$ of this form is typically not strongly convex and the relative strong convexity constant $\mu_{f,A}$ may be zero as illustrated in Example~\ref{example} below.  On the other hand, for a function $f$ of this form it is always the case that $\mu^\star_{f,A}>0$ as Theorem~\ref{thm.loc} below shows.

Theorem~\ref{thm.loc} gives a lower bound for $\mu^\star_{f,A}$ similar in spirit to the lower bound for $\mu_{f,A}$ in Corollary~\ref{prop.bounds}.  The statement and proof of Theorem~\ref{thm.loc} rely on the concept of local facial distance (an extension of the facial distance) introduced in \cite{PenaR16}.  Suppose $\R^m$ is endowed with a norm $\|\cdot\|$.  For $v \in \R^m$ define $\|\cdot\|_v: \R^{m+1} \rightarrow \R$ as follows. For $\bar u =  \matr{u \\ u_{m+1}} \in \R^{m+1}$ let
\begin{equation}\label{eq.quasinorm}
\|\bar u\|_v := \sqrt{\|u\|^2 + \vert\ip{v}{u} + u_{m+1}\vert}.
\end{equation}
Observe that $\|\bar u\|_v > 0$ if $\bar u \ne 0$.  For nonempty $F,G\subseteq \R^{m+1}$ let
\[
\dist_v(F,G) = \dmin_{\bar u\in F, \bar w\in G}\|\bar u-\bar w\|_v.
\]
For $\bar A \in \R^{(m+1)\times n}$ with at least two different columns and $v \in \R^m$ let $$F(v):=\displaystyle\Argmin_{\bar u\in\conv(\bar A)} \left\langle\matr{v\\1},{\bar u}\right\rangle \in \faces(\conv(\bar A)).$$  
The {\em local facial distance} $\Phi_v(\bar A)$ is defined as follows
\[
\Phi_v(\bar A):=\dmin_{G \in \faces(F(v))\atop \emptyset \ne G \ne \conv(\bar A)} \dist_v(G,\conv(\bar A\setminus G)).
\]
Observe that $\Phi_v(\bar A) > 0$ for all $v\in \R^m$ and $\bar A\in \R^{m\times n}$ with at least two different columns.  Furthermore, the facial distance can be recovered a special case of the local facial distance: Given $A\in \R^{m\times n}$ take $\bar{A}=\matr{A\\ 0} \in \R^{(m+1)\times n}$ and $v=0 \in \R^m$.  In this case, $F(v)=\conv(\bar A) =\conv(A) \times \{0\}$, $\faces(\conv(\bar A)) = \{F\times \{0\}: F\in \faces(\conv(A))\}$, and $\|\bar u \|_v = \|u\|$ for all $\bar u = \matr{u\\0}\in \conv(\bar A)$.  Therefore,
$$\Phi_v(\bar{A})= \dmin_{G \in \faces(F(v))\atop \emptyset \ne G \ne \conv(\bar A)} \dist_v(G,\conv(\bar A\setminus G)) =\dmin_{F \in \faces(\conv(A))\atop \emptyset \ne F \ne \conv(A)} \dist(F,\conv( A\setminus F)) = \Phi(A).$$

\begin{theorem}\label{thm.loc} Suppose $\R^n$ is endowed with the $\ell_1$ norm.  Let $A \in \R^{m\times n}, b \in \R^m$, and $E \in \R^{p\times m}$ be such that $\matr{EA\\ b\transp A}$ has at least two different columns. Let $f:\R^m\rightarrow \R \cup\{\infty\}$ be defined by $f(u) = g(Eu) + \ip{b}{u}$ where $g:\R^p \rightarrow \R\cup\{\infty\}$ is $\mu_g$-strongly convex on $\conv(EA)$ for some $\mu_g>0$.  Then $v = \frac{2}{\mu_g}\nabla g(Eu^\star)$ is the same for all $u^\star \in \displaystyle\Argmin_{u\in \conv(A)} f(u)$ and
\begin{equation}\label{eq.main.thm}
\mu_{f,A}^\star \ge \frac{\mu_g \cdot \Phi_v(\bar A)^2}{4} > 0,
\end{equation}
for $\bar A :=\matr{EA\\ \frac{2}{\mu_g}b\transp A}$.
\end{theorem}

Observe that the bound in Theorem~\ref{thm.loc} scales appropriately in the following sense.  Suppose we replace $f(u) = g(Eu) + \ip{b}{u}$ by $
\tilde f(u) := \lambda f(u) = \lambda g(Eu) + \ip{\lambda b}{u} =: \tilde g(Eu) + \langle\tilde b,u\rangle$ for some $\lambda > 0$.  Then $\mu^\star_{\tilde f,A} = \lambda \mu^\star_{f,A}, \; \mu_{\tilde g} = \lambda \mu_g, \; \nabla \tilde g = \lambda \nabla g$ and thus $v$ and $\bar A$ are unchanged.  Therefore all terms in inequality~\eqref{eq.main.thm} scale exactly by $\lambda$.

The following corollary specializes Theorem~\ref{thm.loc} to the special case when the objective function $f$ is a  convex quadratic function.  In that case $\mu_{f,A}^\star$ is positive regardless of the strong convexity of $f$.  

\begin{corollary}\label{corol.mustar} Let $A \in \R^{m\times n}, b\in \R^m$ and $Q\in \R^{m\times m}$ be such that $Q$ is symmetric positive semidefinite and $\matr{Q^{1/2}A\\b\transp A}$ has at least two different columns.  
Let  $f:\R^m\rightarrow \R$ be defined by $f(u) = \frac{1}{2} \ip{Qu}{u} + \ip{b}{u}$.  Then $v = 2Q^{1/2}u^\star$ is the same for all $u^\star \in \displaystyle\Argmin_{u\in \conv(A)} f(u)$ and
$$\mu_{f,A}^\star \ge \frac{\Phi_v(\bar A)^2}{4}>0,$$  
for $\bar A :=\matr{Q^{1/2}A\\ 2b\transp A}$.
\end{corollary}

The next example describes a simple case when $\mu_{f,A}^\star > \mu_{f,A} = 0$.

\begin{example}\label{example}
{\em Let $f:\R^2\rightarrow \R$ be and $A\in \R^{2\times 3}$ be as follows
$$f(s,t) = \frac{1}{2} s^2 + t, \; A = \matr{1 & -1 & 0 \\ 0 & 0 & 1}.$$
For $u = \matr{0 & 0}\transp$ and $x = \matr{0 & 0 & 1}\transp$ we have $f(Ax) - f(u) - \ip{\nabla f(u)}{Ax-u} = 0$ and $\dist(x,Z(u))=2$.  Hence $\mu_{f,A} = 0$.  On the other hand, $u^\star = \matr{0 & 0}\transp$ and 
hence $Z^\star = \{\matr{\frac{1}{2}& \frac{1}{2} & 0}\transp\}.$  Corollary~\ref{corol.mustar} and some straightforward calculations yield the lower bound
\[
\mu_{f,A}^\star \ge \frac{\Phi_0(\bar A)^2}{4} = \dfrac{1}{2},
\]
for $\bar A = \matr{1 & -1 & 0 \\ 0 & 0 & 0 \\ 0 & 0 & 2}$. A more detailed calculation shows that indeed $\mu_{f,A}^\star = \dfrac{1}{2}$.
% is  attained at $x = \matr{0 & 0 & 1}\transp \in \Delta_2\setminus Z^\star$.  Furthermore, since $u^\star = \matr{0 & 0}\transp$, 

}
\end{example}

The proof of Theorem~\ref{thm.loc} relies on Lemma~\ref{lemma.lp.sol.loc} and Lemma~\ref{lemma.error.bound.loc} below.  These lemmas in turn use the following notation.  For $x\in \Delta_{n-1}$, let $I(x):=\{i\in\{1,\dots,n\}: \ip{e_i}{x} > 0\}.$

\begin{lemma}\label{lemma.lp.sol.loc}  Let $\bar{A}\in \R^{(m+1)\times n}$ have at least two different columns and $v\in\R^m$.   Suppose $\bar{u}\in F(v)$ and $x\in \Delta_{n-1}$ are such that $\bar A x \ne \bar u$.  Then for $d:=\frac{\bar{A}x-\bar{u}}{\|\bar{A}x-\bar{u}\|_v}$
\begin{equation}\label{eq.lp.sol.loc}
\Phi_v(\bar{A}) \le \max\{\lambda : \exists y,z\in \Delta_{n-1}, I(y) \subseteq I(x), \bar{A}(y-z) = \lambda d\}.
\end{equation}
\end{lemma}
\begin{proof} %Lemma~\ref{lemma.lp.sol.loc} 
This proof is a modification of the proof of~\cite[Proposition 1]{PenaR16}.  
Let $I:= I(x)$ and $J:=\{j\in\{1,\dots,n\}: \bar a_j \in F(v)\}$.  We will prove the following inequality that evidently implies~\eqref{eq.lp.sol.loc}:  
\begin{equation}\label{eq.lp.sol.loc.strong}
\Phi_v(\bar{A}) \le \max\{\lambda : \exists y,z\in \Delta_{n-1}, I(y) \subseteq I, I(z) \subseteq J, \bar{A}(y-z) = \lambda d\}.
\end{equation}
To that end, observe that  the right-hand side in~\eqref{eq.lp.sol.loc.strong} can be computed via the following primal-dual pair of linear programs
\begin{equation}\label{eq.primal}
\begin{array}{rl}
\dmin_{y_I,z_J,\lambda} & \lambda \\
& \bar A_I y_I - \bar A_J z_J - \lambda d = 0 \\
& \1_I\transp y_I =1\\
& \1_J\transp z_J = 1 \\
& y_I, z_J \ge 0.
\end{array}
\end{equation}
and 
\begin{equation}\label{eq.dual}
\begin{array}{rl}
\dmax_{p,t,\tau} & t + \tau \\
& \bar A_I\transp p \le t\1_I\\
& \bar A_J\transp p \ge -\tau \1_J \\
& \ip{d}{p} = 1.
\end{array}
\end{equation}
Observe that~\eqref{eq.primal} is feasible because in particular the triple $(y_I,z_J,\lambda)$ defined by taking $y_I = x_I$, any $z_J \ge 0$ such that $\bar u = \bar A_J z_J$ and $\1_J\transp z_J = 1$, and $\lambda = \|\bar A x - \bar u\| $ satisfies the constraints in~\eqref{eq.primal}. Furthermore,~\eqref{eq.primal} is evidently bounded since any feasible $(y_I,z_J,\lambda)$ must have $y_I$ and $z_J$ bounded and $\bar A_I y_I -\bar A_J z_J = \lambda  d$ with $d\ne 0$.  Therefore both~\eqref{eq.primal} and~\eqref{eq.dual} attain their finite optimal values.
Let $(\hat y_I,\hat z_J,\hat \lambda)$ and $(\hat p,\hat t,\hat \tau)$ be optimal solutions to~\eqref{eq.primal} and~\eqref{eq.dual} respectively. 
Then $\hat \lambda = \|\bar A_I \hat y_I - \bar A_J \hat z_J\|$.   By  complementary slackness, $\hat z_j > 0$ if and only if $\bar a_j \in \Argmin_{\ell\in J} \ip{\bar a_\ell}{\hat p}$. 
 Likewise, $\hat w_i > 0$ if and only if 
$\bar a_i \in \Argmax_{\ell\in I} \ip{\bar a_\ell}{\hat p}$.
 Therefore $\bar A_Jz_J \in G:= \Argmin_{\bar u\in F}\ip{\hat p}{\bar u} \in \faces(F)$ and $\bar A_I \hat y_I \in \conv(\bar A\setminus G)$.  To finish, observe that
 \begin{align*}
 \Phi_v(\bar A) &\le \dist_v(G,\bar A\setminus G)\\
 &\le \|\bar A_I \hat y_I - \bar A_J\hat z_J\|_v \\
% = \lambda \\
 & = \max\{\lambda : \exists y,z\in \Delta_{n-1}, I(y) \subseteq I, I(z) \subseteq J, \bar{A}(y-z) = \lambda d\}.
 \end{align*}
\end{proof}
%Thus we defer its proof to the Appendix.

\begin{lemma}\label{lemma.error.bound.loc}  Let $\bar{A}\in \R^{(m+1)\times n}$ and $v\in\R^m$.   Then for all $\bar{u}\in F(v)$ and $x\in \Delta_{n-1}$ 
\begin{equation}\label{eq.error.bound.loc}
\min\{\|x-z\|_1: z \in \Delta_{n-1}, \bar{A}z = \bar{u}\} \le \frac{2\|\bar{A}x-\bar{u}\|_v}{\Phi_v(\bar{A})}.
\end{equation}
\end{lemma}
\begin{proof}
Suppose $\bar{A}x \ne \bar{u}$ as otherwise there is nothing to show.  To prove~\eqref{eq.error.bound.loc} we  proceed by contradiction.  Assume
\begin{equation}\label{eq.contra.loc}
z\in \Delta_{n-1}, \bar{A}z = \bar{u} \Rightarrow \|x-z\|_1 > \frac{2\|\bar{A}x-\bar{u}\|_v}{\Phi_v(\bar{A})}.
\end{equation}
Let $d:= \frac{\bar{A}x-\bar{u}}{\|\bar{A}x-\bar{u}\|_v}$
 and consider the following linear program
\begin{equation}\label{eq.lp.loc}
\begin{array}{rl}
\dmax_{w,t} & t \\
& \bar{A}w = t d \\
& x-w \in \Delta_{n-1} \\
&\|w\|_1 \le \dfrac{2t}{\Phi_v(\bar{A})}.
\end{array}
\end{equation}
By~\eqref{eq.lp.sol.loc} there exist $y,z\in \Delta_{n-1}$ with $I(y) \subseteq I(x)$ and $\bar{A}(y-z) = \Phi_v(\bar{A}) d$.  Thus for $\delta > 0$ sufficiently small the linear program~\eqref{eq.lp.loc} has a feasible solution $(w,t) = \delta \cdot(y-z,\Phi_v(\bar{A}))$ with $t = \delta \cdot \Phi_v(\bar{A})>0$.  Assumption~\eqref{eq.contra.loc} thus implies that \eqref{eq.lp.loc} has an optimal solution $(\hat w,\hat t)$ with $0 < \hat t < \|\bar{A}x-\bar{u}\|_v$.  Let $\hat x:= x - \hat w \in \Delta_{n-1}$.  Observe that $\bar{A} \hat x - \bar{u} = \bar{A}x-\bar{u} - \bar{A} \hat w = (\|\bar{A}x-\bar{u}\|_v - \hat t)d \ne 0$ and hence $d = \frac{\bar A \hat x - \bar{u}}{\|\bar{A} \hat x - \bar{u}\|_v}.$  Consider the modification of~\eqref{eq.lp.loc} obtained by replacing $x$ with $\hat x$:
\begin{equation}\label{eq.lp.again.loc}
\begin{array}{rl}
\dmax_{w,t} & t \\
& \bar{A}w = t d \\
& \hat x-w \in \Delta_{n-1} \\
&\|w\|_1 \le \dfrac{2t}{\Phi_v(\bar{A})}.
\end{array}
\end{equation}
Proceeding as above, it follows that~\eqref{eq.lp.again.loc} has a feasible solution $(w',t')$ with $t' > 0$. This implies that $(\hat w + w',\hat t +t')$ is feasible for \eqref{eq.lp.loc} and $\hat t + t' > \hat t$ which contradicts the optimality of $(\hat w,\hat t)$  for \eqref{eq.lp.loc}.
\end{proof}

\begin{proof}[Proof of Theorem~\ref{thm.loc}]
 The optimality conditions for $\dmin_{u \in \conv(A)} f(u)$ imply that for all $u^\star \in \displaystyle\Argmin_{u\in \conv(A)} f(u)$ and all $u \in \conv(A)$
\begin{equation}\label{eq.opt.conds}
\ip{E\transp \nabla g(Eu^\star) + b}{u-u^\star} \ge 0.
\end{equation}
 Therefore if $u^\star,u' \in \displaystyle\Argmin_{u\in \conv(A)} f(u)$, the strong convexity of $g$ and~\eqref{eq.opt.conds} imply
\begin{align*}
\mu_g \|Eu^\star - Eu'\|^2 &\le \ip{\nabla g(Eu^\star) - \nabla g(Eu')}{Eu^\star- Eu'}\\
 &= \ip{E\transp \nabla g(Eu^\star) - E\transp \nabla g(Eu')}{u^\star- u'}  \\ &\le 0.
\end{align*}
Hence $Eu^* = Eu'$ whenever $u^\star,u' \in \displaystyle\Argmin_{u\in \conv(A)} f(u)$.  In particular, $v = \nabla g(Eu^\star)$ is the same for all $u^\star \in \displaystyle\Argmin_{u\in \conv(A)} f(u)$. Furthermore, from~\eqref{eq.opt.conds} it follows that $\bar u := \bar A z$ is the same for all $z\in Z^\star$ and $\bar u \in F(v)$. Lemma~\ref{lemma.error.bound.loc} implies that for all $x\in \Delta_{n-1}$ 
\begin{equation}\label{eq.norm}
\dist(x,Z^\star) = \min\{\|x-z\|: z\in\Delta_{n-1},\;\bar A z = \bar u \} \le \frac{2 \|\bar A x - \bar u\|_v}{\Phi_v(\bar A)}.
\end{equation}
Next, observe that the strong convexity of $g$ and~\eqref{eq.opt.conds} imply that for all $ x \in \Delta_{n-1}$ and $z\in Z^\star$
\begin{align}\label{eq.strong.conv}
f(Ax)-f^\star & = g(EAx) - g(EAz) + \ip{b}{Ax-Az} \notag\\
& \ge \ip{\nabla g(EAz)}{EAx-EAz} + \frac{\mu_g}{2}\|EAx-EAz\|^2 + \ip{b}{Ax-Az} \notag\\
& = \frac{\mu_g}{2} \left(\|EAx-EAz\|^2 
+ \frac{2}{\mu_g}\ip{E\transp\nabla g(EAz)+b}{Ax-Az}
\right) \notag\\
& = \frac{\mu_g}{2} \left(\|EAx-EAz\|^2 
+ \left\vert \ip{v}{EAx-EAz} + \ip{\frac{2}{\mu_g}b}{Ax-Az}\right\vert
\right) \notag\\
& = \frac{\mu_g}{2} \|\bar A x - \bar u\|_v^2.
\end{align}
Putting together~\eqref{eq.norm} and~\eqref{eq.strong.conv}  we get
\[
\mu^\star_{f,A} = \min_{ x\in \Delta_{n-1}\setminus Z^\star} \frac{2(f(Ax) - f^\star)}{\dist(x,Z^\star)^2} \ge 
\min_{ x\in \Delta_{n-1}\atop Ax \ne \bar u} \frac{\mu_g\cdot\|\bar A x - \bar u\|_v^2}{\dist(x,Z^\star)^2} \ge \frac{\mu_g \cdot\Phi_v(\bar A)^2}{4}.
\]
\end{proof}

\subsection{Proofs of Proposition~\ref{prop.diam.dist} and Proposition~\ref{prop.dist.gral}}\label{sec.proofs}

\begin{proof}[Proof of Proposition~\ref{prop.diam.dist}.] For $\diam(A)$ observe that
\begin{align*}
\diam(A) &= \max_{x,y\in\Delta_{n-1}} \|A(x-y)\|
%\\&
= \max_{x,y\in\Delta_{n-1}\atop \|x-y\|_1=2} \frac{2\|A(x-y)\|}{\|x-y\|_1} 
%\\&
= \max_{x,y\in\Delta_{n-1}\atop x\ne y} \frac{2\|A(x-y)\|}{\|x-y\|_1} 
\\&
= \max_{u \in \conv(A)\atop x\in\Delta_{n-1}\setminus Z(u)} \frac{2\|Ax-u\|}{\dist(x,Z(u))}. 
\end{align*} 
For $\Phi(A)$ we prove the two inequalities separately.  From Lemma~\ref{lemma.error.bound.loc} applied to $\bar A = \matr{A\\0}$ and $v = 0$ it follows that for all $u \in \conv(A)$ and $x\in \Delta_{n-1}$
\[
\dist(x,Z(u)) = \min\{\|x-z\|: z\in\Delta_{n-1},\; A z = u \} \le \frac{2\|Ax-u\|}{\Phi(A)}.
\]
Therefore $\Phi(A) \le \dmin_{u \in \conv(A)\atop x\in\Delta_{n-1}\setminus Z(u)}\frac{2\|Ax-u\|}{\dist(x,Z(u))}.$ For the reverse inequality, let $F\in \faces(\conv(A))$ be such that 
$\emptyset \ne F \ne \conv(A)$ and $\Phi(A) = \dist(F,\conv(A\setminus F))$.  Then $\Phi(A) = \|A\hat x  - \hat u\|$ for some $\hat u \in F$ and $\hat x \in \Delta_{n-1}$ with $A\hat x\in \conv(A\setminus F).$  Since 
$A\hat x\in \conv(A\setminus F)$, without loss of generality we may assume that $\hat x\in \Delta_{n-1}$ is chosen so that $i\in I(\hat x) \Rightarrow a_i\not \in F.$  Since $F$ is a face, it follows that $Az \in F \Rightarrow I(z) \cap I(\hat x) = \emptyset.$    In particular $\dist(\hat x,Z(\hat u)) = 2$ and thus
\[
\dmin_{u \in \conv(A)\atop x\in\Delta_{n-1}\setminus Z(u)}\frac{2\|Ax-u\|}{\dist(x,Z(u))} \le \frac{2\|A\hat x-\hat u\|}{\dist(\hat x,Z(\hat u))} = \|A\hat x-\hat u\| = \Phi(A).
\]
\end{proof}  

\begin{proof}[Proof of Proposition~\ref{prop.dist.gral}]  
For~\eqref{eq.general.diam} observe that \[
\max_{u\in \conv(A)\atop x\in \Delta_{n-1}\setminus Z(u)} \frac{\|Ax - v\|}{\dist(x,Z(u))} = 
\max_{x,y\in\Delta_{n-1}\atop x\ne y} \frac{\|A(x-y)\|}{\|x-y\|}
=
\max_{w\in\R^n\setminus\{0\}\atop\1\transp w = 0 } \frac{\|Aw\|}{\|w\|}.
\]
For~\eqref{eq.general.Phi}, observe that for all $x,y\in \Delta_{n-1}$ we have $x-y = \sum_{i=1}^n \ip{e_i}{x-y}e_i$ and hence
\[
\|x - y\| %= \left\|\sum_{i=1}^n (x_i-y_i)e_i \right\| 
\le \sum_{i=1}^n \vert \ip{e_i}{x-y} \vert \|e_i\| \le \dmax_{i=1,\dots,n} \|e_i\| \cdot \|x-y\|_1.
\]
Thus Proposition~\ref{prop.diam.dist} yields
\[
\dmin_{u\in \conv(A)\atop x\in \Delta_{n-1}\setminus Z(u)} \frac{2\|Ax - u\|}{\dist(x,Z(u))} \ge \dmin_{u\in \conv(A)\atop x\in \Delta_{n-1}\setminus Z(u)} \frac{2\|Ax - u\|}{\dmax_{i=1,\dots,n} \|e_i\| \cdot \dist_1(x,Z(u))} = \frac{\Phi(A)}{\dmax_{i=1,\dots,n} \|e_i\|}.
\]

\end{proof}  
%We conclude this section with the proof of Proposition~\ref{prop.diam.dist}.  

%%%%%%%%%%%%%%%%%%%%%%%%%%%%%%%%%%%%%%%%%%%%%%%%%%%%%%%%%%%%%%%%%%%%%%%%%%%%%%%%%%%%%%%
%
%New Convergence Results for Polytope Constrained Convex  Optimization Algorithms
%
%%%%%%%%%%%%%%%%%%%%%%%%%%%%%%%%%%%%%%%%%%%%%%%%%%%%%%%%%%%%%%%%%%%%%%%%%%%%%%%%%%%%%%%

\section{Linear convergence of first-order methods}
%Frank-Wolfe and projected gradient methods}
\label{sec.algos}

This section discusses linear convergence results for two first-order algorithms for the problem \eqref{eq.Prob} namely the Frank-Wolfe with away steps (Algorithm~\ref{alg.FWA}) and the projected gradient method (Algorithm~\ref{alg.proj.grad}).  
Linear convergence results for both algorithms have been previously established in~\cite{BeckS15, BeckT09, Nest13, PenaR16, LacoJ15} under suitable assumptions.   
The goal of this section is to illustrate the role of the relative condition number $\frac{L_{f,A}}{\mu^\star_{f,A}}$ in these linear convergence results.  
The role of the relative condition number $\frac{L_{f,A}}{\mu^\star_{f,A}}$ is akin to the role of the usual condition number $\frac{L_f}{\mu_f}$ in the linear convergence of the gradient descent algorithm for unconstrained convex minimization.

Both proofs of linear convergence rely on the following elementary observation.  If $a \le 0, \; b>0,$ and $\alpha_{\max} > 0$ then
\begin{equation}\label{quad.model}
\min_{\alpha \in [0,\alpha_{\max}]} \, a \alpha + \frac{b}{2} \alpha^2 = 
 \left\{ \begin{array}{ll} -\frac{a^2}{2b} & \text{if} \; \alpha_{\max} > -\frac{a}{b}\\ 
\alpha_{\max}\left(a+\frac{b}{2}\alpha_{\max} \right) \le \frac{a}{2} \alpha_{\max}& \text{if} \; \alpha_{\max} \le -\frac{a}{b}.\end{array} \right.
\end{equation}

\subsection{Frank-Wolfe algorithm with away steps}
Algorithm~\ref{alg.FWA} gives a description of the Frank-Wolfe algorithm with away steps for \eqref{eq.Prob}.  This version of the algorithm has been previously discussed in~\cite{BeckS15, PenaR16, LacoJ15}. 
Algorithm~\ref{alg.FWA} relies on the following notation.  Given $u = Ax \in \conv(A)$, let $I(x):=\{i\in \{1,\dots,n\}: \ip{e_i}{x} > 0\}$.  The set $I(x)$ describes the {\em support} of $u = Ax$, that is, the indices of the columns of $A$ that appear with positive weight in the convex combination $u = Ax$.

\begin{algorithm}
  \caption{Frank-Wolfe algorithm with away steps
    \label{alg.FWA}}
  \begin{algorithmic}[1]
\State Pick $x_0 \in \Delta_{n-1}$; put $u_0:= Ax_0;\; k:=0$
 \For{$k=0,1,2,\dots$}
\State $j := \displaystyle\argmin_{i=1,\dots,n} \ip{\nabla f(u_k)}{a_i};\; \ell := \displaystyle\argmax_{i\in I(x_k)}  \ip{\nabla f(u_k)}{a_i}$
 \If{$\ip{\nabla f(u_k)}{a_j-u_k} < \ip{\nabla f(u_k)}{u_k-a_{\ell}}$ or $\vert I(x_k)\vert =1$}
 % (regular step)
\State $v:=a_j - u_k;\; w:=e_j - x_k; \; {\alpha}_{\max} := 1\;\;\;\;\,$  \quad (regular step)
\quad \Else 
$\;\;$
\State  $v:= u_k - a_\ell; \; {w}:= x_k -e_{\ell}; \; \alpha _{\max} := \frac{\ip{e_\ell}{x_k}}{1-\ip{e_\ell}{x_k}}\;\;$  (away step) 
\EndIf 
\State  choose $\alpha _k  \in [0,\alpha_{\max}]$
\State $x_{k+1} := x_k + \alpha _k w ; \; \; u_{k+1} := u_k + \alpha _k v = Ax_{k+1}$ 
\EndFor
  \end{algorithmic}
\end{algorithm}
A critical detail in Algorithm~\ref{alg.FWA} is the choice of step size $\alpha_k$ in Step 9.  The construction of $L_{f,A}$ implies that for $\alpha \in [0,\alpha_{\max}]$
\[
f(u_k + \alpha v) \le f(u_k) + \alpha \ip{\nabla f(u_k)}{v} + \frac{L_{f,A} \alpha^2}{2} \|w\|_1^2 \le f(u_k) + \alpha \ip{\nabla f(u_k)}{v} + 2 L_{f,A} \alpha^2.
\] 
To simplify  our analysis of linear convergence we will assume that $\alpha_k$ in Step 9 is chosen via 
\begin{equation}\label{eq.step.size}
\alpha_k := \argmin_{\alpha \in [0,\alpha_{\max}]}
\left\{f(u_k) + \alpha \ip{\nabla f(u_k)}{v} + 2L_{f,A} \alpha^2 
\right\} 
=\min\left\{\alpha_{\max},-\frac{\ip{\nabla f(u_k)}{v}}{4 L_{f,A}} \right\}.
\end{equation}
This is evidently possible only in the ideal case when $L_{f,A}$ is known.  In the more realistic case when $L_{f,A}$ is not known, a standard backtracking procedure can be used to choose a constant $L > 0$ bounded above by a constant multiple of $L_{f,A}$ and such that the step size
$$
\alpha_k := \argmin_{\alpha \in [0,\alpha_{\max}]}\left\{f(u_k) + \alpha \ip{\nabla f(u_k)}{v} + 2L \alpha^2 
\right\}=\min\left\{\alpha_{\max},-\frac{\ip{\nabla f(u_k)}{v}}{4 L} \right\} 
$$
satisfies
\[
f(u_k + \alpha_k v) \le f(u_k) + \alpha_k \ip{\nabla_k f(u_k)}{v} + 2 L \alpha_k^2.
\]
Our ensuing analysis would then apply with $L_{f,A}$ replaced by a constant multiple of it.  For the remainder of this subsection we will assume that $\alpha_k$ is indeed chosen via~\eqref{eq.step.size}.  Combining this assumption and~\eqref{quad.model} applied to $a = \ip{\nabla f(u_k)}{v}$ and $b = 4L_{f,A}$ we obtain
\begin{equation}\label{eq.FWA.dec}
f(u_{k+1}) - f(u_k) \le 
 \left\{ \begin{array}{ll} -\frac{\ip{\nabla f(u_k)}{v}^2}{8L_{f,A}} & \text{if} \; \alpha_k < \alpha_{\max}\\ 
 \frac{\ip{\nabla f(u_k)}{v}}{2} \alpha_{\max}& \text{if} \; \alpha_k = \alpha_{\max}.\end{array} \right.
\end{equation}

 The following result provides the crux of the linear convergence of 
Algorithm~\ref{alg.FWA}.

\begin{lemma}\label{prop.non-drop} Suppose $\R^n$ is endowed with the $\ell_1$ norm $\|\cdot\|_1$, and $A\in\R^{m\times n}$ and $f:\R^n\rightarrow \R\cup\{\infty\}$ are such that $\mu_{f,A}^\star > 0.$ 
Then the direction $v$ chosen in Step 5 or Step 7 of Algorithm~\ref{alg.FWA} satisfies
\begin{equation}\label{eq.fwa.eq1}
\ip{\nabla f(u_k)}{v}^2 \ge \frac{\mu_{f,A}^\star}{2}(f(u_k) - f^\star)
\end{equation}
and
\begin{equation}\label{eq.fwa.eq2}
\ip{\nabla f(u_k)}{v} \le  f^\star-f(u_k).
\end{equation}
\end{lemma}
\begin{proof}  
The choice of regular versus away steps in Step 5 and Step 7 imply that if $u_k = Ax_k \in \conv(A)$ then
\[
2\ip{\nabla f(u_k)}{v} \le \min_{i=1,\dots,n} \ip{\nabla f(u_k)}{a_i} - \max_{i \in I(x_k)} \ip{\nabla f(u_k)}{a_i} \le 0.
\]
Let $z^\star \in Z^\star$ be such that $\|x_k-z^\star\|_1 = \dist(x_k,Z^\star)$ and $u^\star=Az^\star$.   
Observe that $x_k-z^\star = \delta(z-y)$ where $z,y\in \Delta_{n-1},\; I(z) \subseteq I(x_k),$ and $\delta:=\frac{\|x_k-z^\star\|_1}{2}\le 1$. Thus
\begin{align*}
\ip{\nabla f(u_k)}{u_k-u^\star} &= \ip{\nabla f(u_k)}{A(x_k-z^\star)} 
%\\&
= \delta \ip{\nabla f(u_k)}{A(z-y)} \\
&\le \delta\left(\max_{z\in\Delta_{n-1}\atop I(z)\subseteq I(x_k)} \ip{\nabla f(u_k)}{Az}- \min_{y \in \Delta_{n-1}} \ip{\nabla f(u_k)}{Ay}\right) \\
& = \delta\left(\max_{i \in I(x_k)} \ip{\nabla f(u_k)}{a_i}- \min_{i=1,\dots,n} \ip{\nabla f(u_k)}{a_i}\right) \\
&\le 2\delta \vert\ip{\nabla f(u_k)}{v}\vert.
\end{align*}
The construction of $\mu_{f,A}^\star$, convexity of $f$, and the latter inequality yield
\[
 0 \le \mu_{f,A}^\star \le \frac{2(f(u_k)-f^\star)}{\dist(z,Z^\star)^2} = \frac{f(u_k)-f^\star}{2\delta^2} \le \frac{\ip{\nabla f(u_k)}{u_k-u^\star}}{2\delta^2} \le \frac{|\ip{\nabla f(u_k)}{v}|}{\delta}.
\]
Therefore~\eqref{eq.fwa.eq1} follows.  On the other hand,  the choice of $v$ and convexity of $f$ yields
\[
\ip{\nabla f(u_k)}{v} \le \min_{u\in\conv(A)} \ip{\nabla f(u_k)}{u-u_k} \le \ip{\nabla f(u_k)}{u^\star-u_k} \le f^\star - f(u_k)
\]
and thus~\eqref{eq.fwa.eq2} follows as well.
\end{proof}
Once we are equipped with Lemma~\ref{prop.non-drop}, the following linear convergence result readily follows via a clever counting argument introduced in~\cite{LacoJ15} and subsequently used in~\cite{BeckS15,PenaR16}.  
To provide a full picture of this linear convergence result, the proof below briefly replicates the necessary material from~\cite{BeckS15,LacoJ15,PenaR16}.  

\begin{proposition}\label{thm.FWA}
Suppose $\R^n$ is endowed with the $\ell_1$ norm $\|\cdot\|_1$, and $A\in\R^{m\times n}$ and $f:\R^n\rightarrow \R\cup\{\infty\}$ are such that $\frac{L_{f,A}}{\mu_{f,A}^\star} <\infty.$ If $x_0$ is a vertex of $\Delta_{n-1}$ then the iterates generated by Algorithm~\ref{alg.FWA} satisfy
\[
f(u_k) - f^\star \le \left(1-
\min\left\{\frac{\mu_{f,A}^\star}{16L_{f,A}},\frac{1}{2}\right\}
\right)^{k/2} (f(u_0)-f^\star).
\]
\end{proposition}
\begin{proof} Consider separately the three possible cases that can occur at iteration $k$, namely $\alpha_k < \alpha_{\max}$, $\alpha_k = \alpha_{\max} \ge 1,$ and $\alpha_k = \alpha_{\max} < 1.$

\medskip

\noindent
{\bf Case 1:} $\alpha_k < \alpha_{\max}$. In this case $|I(x_{k+1})| \le |I(x_k)|+1$.  Furthermore,~\eqref{eq.FWA.dec} and~\eqref{eq.fwa.eq1} imply that
\[
f(u_{k+1}) - f(u_k) \le -\frac{\mu^\star_{f,A}}{16 L_{f,A}}(f(u_k) - f^\star).
\]

\medskip

\noindent
{\bf Case 2:} $\alpha_k = \alpha_{\max} \ge 1$.  In this case $|I(x_{k+1})| \le |I(x_k)|$.  Furthermore,~\eqref{eq.FWA.dec} and~\eqref{eq.fwa.eq2} imply that
\[
f(u_{k+1}) - f(u_k) \le -\frac{1}{2}(f(u_k) - f^\star).
\]

\medskip

\noindent
{\bf Case 3:} $\alpha_k = \alpha_{\max} < 1$.  In this case $|I(x_{k+1})| \le |I(x_k)|-1$.  Furthermore,~\eqref{eq.FWA.dec} implies that
\[
f(u_{k+1}) - f(u_k) \le 0.
\]
Therefore to finish it suffices to show that in the first $k$ iterations Case 3 can occur at most $k/2$ times.  Since $|I(x_0)| = 1$ and $|I(x_i)| \ge 1$ for $i=1,2,\dots,$ it follows that for each iteration when Case 3 occurred
%,  and thus the size of $I(x_k)$ decreased by at least one, 
there must have been at least one previous iteration when Case 1 occurred. % so that the size of $I(x_k)$ increased by at least one as well.  
Hence in the first $k$ iterations Case 3 could occur at most $k/2$ times.
\end{proof}

%
%Local Smooth Frank-Wolfe
%
\subsection{Projected gradient}

Algorithm~\ref{alg.proj.grad} gives a description of the projected gradient algorithm for \eqref{eq.Prob}.  This version can be seen as a particular case of more general gradient schemes like those discussed in~\cite{BeckT09,Nest13}.

\begin{algorithm}
  \caption{Projected gradient algorithm
    \label{alg.proj.grad}}
  \begin{algorithmic}[1]
\State Pick $x_0 \in \Delta_{n-1}$; 
 \For{$k=0,1,2,\dots$}
 \State choose $L > 0$ 
\State $x_{k+1} = \dargmin_{x\in \Delta_{n-1}}\left\{ f(Ax_k)+\ip{\nabla f(Ax_k)}{A(x-x_k)} + \frac{L}{2}\|x-x_k\|_2^2\right\}$
\EndFor
  \end{algorithmic}
\end{algorithm}

Like the choice of step size $\alpha_k$ in Step 9 of Algorithm~\ref{alg.FWA}, the choice of $L$ in Step 3 is a critical detail in Algorithm~\ref{alg.proj.grad}. The construction of $L_{f,A}$ implies that 
\[
f(Ax) \le f(Ax_k) + \ip{\nabla f(Ax_k)}{A(x-x_k)} + \frac{L_{f,A}}{2}\|x-x_k\|_2^2.
\]
To simplify our analysis of linear convergence we will assume that Step 3 chooses  $L = L_{f,A}.$  This is possible only if $L_{f,A}$ is known.  In the more realistic case when $L_{f,A}$ is not known, a standard backtracking procedure can be used to choose $L$  bounded above by a constant multiple of $L_{f,A}$ and such that the next iterate $x_{k+1} \in \Delta_{n-1}$ chosen at Step 4 satisfies
\[
f(Ax_{k+1}) \le  f(Ax_k) + \ip{\nabla f(Ax_k)}{A(x_{k+1}-x_k)} + \frac{L}{2}\|x_{k+1}-x_k\|_2^2.
\]
Our ensuing analysis would then apply with $L_{f,A}$ replaced by a constant multiple of it.
The  assumption that $L = L_{f,A}$ in Step 3 of Algorithm~\ref{alg.proj.grad} combined with~\eqref{quad.model} readily imply that for all $z\in \Delta_{n-1}$ such that $\ip{\nabla f(Ax_k)}{A(z-x_k)} \le 0$
\begin{align}
f(Ax_{k+1}) - f(Ax_k) &\le \min_{\alpha \in [0,1]} 
\left(\ip{\nabla f(Ax_k)}{A(z-x_k)}\alpha + \frac{L_{f,A}}{2}\|z-x_k\|_2^2\alpha^2\right)\notag \\
& \le \left\{ 
 \begin{array}{ll} -\frac{\ip{\nabla f(Ax_k)}{A(z-x_k)}^2}{2L_{f,A}\|z-x_k\|^2} & \text{if} \; -\frac{\ip{\nabla f(Ax_k)}{A(z-x_k)}}{L_{f,A}\|z-x_k\|^2} < 1\\ 
 \frac{\ip{\nabla f(Ax_k)}{A(z-x_k)}}{2} & \text{if} \; -\frac{\ip{\nabla f(Ax_k)}{A(z-x_k)}}{L_{f,A}\|z-x_k\|^2} \ge 1.\end{array}
\right.\label{eq.proj.grad.dec}
\end{align}

\begin{proposition}\label{thm.simple}
Suppose $\R^n$ is endowed with the $\ell_2$ norm $\|\cdot\|_2$, and $A\in\R^{m\times n}$ and $f:\R^n\rightarrow \R\cup\{\infty\}$ are such that $\frac{L_{f,A}}{\mu_{f,A}^\star} <\infty.$
 Then the sequence of iterates $\{x_k, \; k=0,1,\dots\}$ generated by Algorithm~\ref{alg.proj.grad} satisfy
\begin{equation}\label{eq.simple.rate}
f(Ax_k) - f^\star \le \left(1- \min\left\{\frac{\mu^\star_{f,A}}{4L_{f,A}},\frac{1}{2}\right\}\right)^k (f(Ax_0) - f^\star).
\end{equation}
\end{proposition}

\begin{proof}
At iteration $k$ let $z\in Z^\star$ be such that $\|z-x_k\|_2^2\le \dfrac{2(f(Ax_k) - f^\star)}{\mu^\star_{f,A}}$.  The convexity of $f$ yields
\[
\ip{\nabla f(Ax_k)}{A(z - x_k)} \le f^\star - f(Ax_k) \le 0
\]
and
\[
\frac{\ip{\nabla f(Ax_k)}{A(z - x_k)}^2}{\|z-x_k\|_2^2} \ge \frac{\mu_{f,A}^\star}{2}(f(Ax_k) - f^\star)
\]
Combining these inequalities and~\eqref{eq.proj.grad.dec} it follows that
\[
f(Ax_{k+1}) - f(Ax_{k}) \le -\min\left\{\frac{\mu^\star_{f,A}}{4L_{f,A}},\frac{1}{2}\right\}(f(Ax_k) - f^\star).
\]
Therefore~\eqref{eq.simple.rate} follows by induction.
\end{proof}

\section*{Acknowledgements}

This research has been  funded by NSF grant CMMI-1534850.

\bibliographystyle{plain}
%\bibliography{RelConditionRefs}

\end{document}